\documentclass[12pt, reqno]{amsart}
\usepackage{latexsym, url}
\usepackage{amsthm}
\usepackage{amsmath}
\usepackage{amsfonts}
\usepackage{amssymb}
\usepackage[hidelinks]{hyperref}
\usepackage[dvips]{graphicx}
\usepackage{xypic}
\usepackage{color}
\input xy
\xyoption{all}

\newtheoremstyle{repeat}{}{}{\itshape}{}{\bfseries}{.}{.5em}{#3, repeated}

\newtheorem{theo}{Theorem}[section]
\newtheorem{lemma}[theo]{Lemma}

\newtheorem{propo}[theo]{Proposition}
\newtheorem{defi}[theo]{Definition}
\newtheorem{coro}[theo]{Corollary}
\newtheorem{rem}[theo]{Remark}

\newtheorem{exam}[theo]{Example}
\newtheorem{exams}[theo]{Examples}
\newtheorem{fact}[theo]{Fact}
\newtheorem{question}[theo]{Question}

\theoremstyle{repeat}
\newtheorem*{repeated-theorem}{Repeat}

\newcommand\NSOP{\operatorname{NSOP}}

\newcommand\Set{\operatorname{\bf Set}}

\newcommand\Bil{\operatorname{\bf Bil}}
\newcommand\BinFunc{\operatorname{\bf BinFunc}}
\newcommand\TwoGraph{\operatorname{\bf TwoGraph}}

\newcommand\Pos{\operatorname{\bf Pos}}
\newcommand\Gra{\operatorname{\bf Gra}}

\newcommand\Grp{\operatorname{\bf Grp}}

\newcommand\Ins{\operatorname{\bf Ins}}

\newcommand\colim{\operatorname{colim}}

\newcommand\gtp{\operatorname{gtp}}

\newcommand\cc{\mathcal {C}}

\newcommand\ce{\mathcal {E}}

\newcommand\ck{\mathcal {K}}

\newcommand\cm{\mathcal {M}}

 \newbox\noforkbox \newdimen\forklinewidth
\forklinewidth=0.3pt \setbox0\hbox{$\textstyle\smile$}
\setbox1\hbox to \wd0{\hfil\vrule width \forklinewidth depth-2pt
 height 10pt \hfil}
\wd1=0 cm \setbox\noforkbox\hbox{\lower 2pt\box1\lower
2pt\box0\relax}
\def\unionstick{\mathop{\copy\noforkbox}\limits}
\newcommand{\nf}{\unionstick}

\newcommand{\ld}{\textup{ld}}
\newcommand{\sld}{\textup{sld}}
\newcommand{\isid}{\textup{isi-d}}
\newcommand{\isif}{\textup{isi-f}}
\newcommand{\slf}{\textup{slf}}

\date{1 July 2024}
\subjclass{18C35 (Primary), 03C45, 03C48, 03C52 (Secondary)}
\keywords{locally presentable category; locally multipresentable category; accessible category; abstract elementary category; stable independence; simple independence; NSOP$_1$-like independence}
 
\begin{document}
\title[Unstable independence from the categorical point of view]
{Unstable independence from the categorical point of view}
\author[M. Kamsma and J. Rosick\'{y}]
{M. Kamsma and J. Rosick\'{y}}
\thanks{The first author is supported by the EPSRC grant EP/X018997/1 and the second author is supported by the Grant Agency of the Czech Republic under the grant 22-02964S}
\address{
\newline M. Kamsma\newline
School of Mathematical Sciences\newline
Queen Mary University of London, Faculty of Science and Engineering\newline
327 Mile End Road, London E1 4NS, United Kingdom\newline
mark@markkamsma.nl\newline
\newline J. Rosick\'{y}\newline
Department of Mathematics and Statistics\newline
Masaryk University, Faculty of Sciences\newline
Kotl\'{a}\v{r}sk\'{a} 2, 611 37 Brno, Czech Republic\newline
rosicky@math.muni.cz
}

\begin{abstract}
We give a category-theoretic construction of simple and NSOP$_1$-like independence relations in locally finitely presentable categories, and in the more general locally finitely multipresentable categories. We do so by identifying properties of a class of monomorphisms $\cm$ such that the pullback squares consisting of morphisms in $\cm$ form the desired independence relation. This generalizes the category-theoretic construction of stable independence relations using effective unions or cellular squares by M. Lieberman, S. Vasey and the second author to the unstable setting.
\end{abstract} 

\maketitle

\tableofcontents

\section{Introduction}
Stability theory is an important branch of model theory that was largely created by S.\ Shelah \cite{Sh}, and it aims to classify structures based on their logical complexity. Its central tool is that of independence relations. Examples of independence relations are linear independence in vector spaces and algebraic independence in fields. In the now-dominant anchor notation introduced by Makkai \cite{M}, this is rendered as a relation $A \nf_C^M B$ on quadruples of sets understood to mean that $A$ is independent from $B$ over $C$ in $M$. Shelah identified an important class of well-behaved theories, called \emph{stable} theories. He developed independence relations for stable theories through the notions of forking and dividing \cite{Sh}. Later, this work was generalised to the broader class of \emph{simple} theories \cite{KP}. More recently, it was found that \emph{NSOP$_1$} theories---the class of NSOP$_1$ theories properly contains the simple theories---admit a good independence relation, by relaxing the notions of forking and dividing to Kim-forking and Kim-dividing \cite{KaplanR}.

In this paper we concern ourselves with a category-theoretic approach to model-theoretic independence relations, through the framework of accessible categories. An accessible category is, intuitively, a category with all sufficiently directed co\-li\-mits and such that every object can be seen as a sufficiently directed colimit of \emph{small} objects where smallness is interpreted as a notion of size which make sense in an arbitrary category. Examples are categories of models of a first-order theory with elementary embeddings as morphisms or the much more general setting of AECs (Abstract Elementary Classes) with strong embeddings. Accessible categories having all colimits are called locally presentable. If ``small'' means \emph{finite} we get locally finitely presentable categories. All this can be found in \cite{MP,AR}. Doing model theory in accessible categories is in fact very close to the framework of AECs \cite{BR}. These frameworks allow for a different and more general approach to model theory, which has been used to obtain deep model-theoretic results (see e.g., \cite{ShAEC, GM, BGKV, HK, MR}).

Focusing on independence relations, recent developments \cite{LRV, LRV1, K, K1} allow for a category-theoretic treatment of them as a calculus of commuting squares. Even, canonicity theorems have been established. They state that in a given category there can be at most one nice enough independence relation (see e.g.\ Theorem \ref{canonicity}). This then recovers part of the classical \emph{stability hierarchy}, based on which kind of independence relation exists in a given category: in increasing generality we have stable, simple and NSOP$_1$-like independence relations, corresponding  to the similarly named classes of theories in classical stability theory.

Stable independence relations satisfy the uniqueness property, which roughly states that there is a unique way to complete a span of arrows $A \leftarrow C \to B$ to an independent square. This uniqueness is taken up to an appropriate equivalence relation, whose equivalence classes are called \emph{Galois types}. The distinguishing property between simple and stable independence is that the uniqueness property is replaced by the weaker \emph{$3$-amalgamation}, a higher-dimensional amalgamation axiom for independent squares. The difference between simple and NSOP$_1$-like independence is that the latter does not satisfy the base monotonicity property (whereas the former does).

In \cite{LRV, LRV1} a stable independence relation is constructed category-theoretically in locally finitely presentable categories. In \cite{LRV} this is done by considering effective unions. These are pullback squares of regular monomorphisms such that the induced morphism from the relevant pushout square is also a regular monomorphism. By replacing the class of regular monomorphisms with some class $\cm$ of monomorphisms, we get the notion of cellular squares, and \cite{LRV1} generalizes the above construction of stable independence using cellular squares. The paper \cite{LRV1} also relates stable independence to the homotopy-theoretic notion of cofibrant generation.

In this paper we construct NSOP$_1$-like and simple independence relations category-theoretically, similar to how \cite{LRV, LRV1} constructs stable independence relations. We do this by essentially replacing the uniqueness property by $3$-amalgamation. Since the uniqueness property comes directly from a pushout being `hidden' in the independent squares, we drop this requirement and work instead with a higher-dimensional requirement that we call \emph{cubic}. This results in pullback squares forming an NSOP$_1$-like independence relation. Note that we have also lost base monotonicity, because in the stable case that follows from uniqueness. We therefore also identify a further propery that corresponds to base monotonicity, so that property holds exactly when this NSOP$_1$-like independence relation is simple.

\textbf{Main results.}
We establish conditions on a class of monomorphisms $\cm$ in a category $\ck$ such that the category $\ck_\cm$ admits an NSOP$_1$-like or simple independence relation. Here $\ck_\cm$ is the category whose objects are those of $\ck$ and whose morphisms are those from $\cm$.
\begin{theo}
\label{nsop}
Suppose that $\ck$ is a locally finitely presentable category equipped with a cubic, continuous and accessible class $\cm$ of monomorphisms. Then $\ck_\cm$ has an $\NSOP_1$-like independence relation given by pullback squares.
\end{theo}
\begin{theo}
\label{simple}
Suppose that $\ck$ is a locally finitely presentable category equipped with a strongly cubic and continuous class $\cm$ of monomorphisms. Then $\ck_\cm$ has a simple independence relation given by pullback squares.
\end{theo}
We then generalize the above theorems to locally multipresentable categories, which are accessible categories that have all connected limits.
\begin{theo}
\label{lmp-nsop}
Suppose that $\ck$ is a locally finitely multipresentable category equipped with a multicubic, continuous and accessible class $\cm$ of monomorphisms. Then $\ck_\cm$ has an $\NSOP_1$-like independence relation given by pullback squares.
\end{theo}
\begin{theo}
\label{lmp-simple}
Suppose that $\ck$ is a locally finitely multipresentable category equipped with a strongly multicubic, continuous and accessible class $\cm$ of monomorphisms. Then $\ck_\cm$ has an simple independence relation given by pullback squares.
\end{theo}
The canonicity theorems (Theorem \ref{canonicity}) for categorical independence relations immediately apply to simple independence, but for NSOP$_1$-like independence we need a technical condition called \emph{the existence axiom for isi-forking}. We prove exactly that in Corollary \ref{existence-axiom-presentable}. In fact, we prove more generally that the existence of a weakly stable independence relation is enough (Theorem \ref{existence-axiom}). We can thus summarize the canonicity of the above independence relations as follows.
\begin{coro}\label{canonicity-in-main-thms}
The independence relations in Theorems \ref{nsop}, \ref{simple} and \ref{lmp-simple} are canonical in the following sense. Given a category $\ck$ and class of monomorphisms $\cm$ satisfying the conditions in one of these theorems, then any NSOP$_1$-like independence relation on $\ck_\cm$ is given by pullback squares.
\end{coro}

\section{AECats and independence relations}
We recall a simplified version of the framework of \emph{AECats} (\emph{Abstract Elementary Categories}) from \cite{K}, where these are defined as a pair of categories $(\cc, \cm)$. We will only be interested in the case $\cc = \cm$, hence the simplification.
\begin{defi}
{
\em
An \emph{AECat} is an accessible category with directed colimits and where all morphisms are monomorphisms.
}
\end{defi}
Throughout we will only be interested in AECats with the amalgamation property.
\begin{defi}
{
\em
We say that a category has the \emph{amalgamation property} or \emph{AP} if any span of morphisms can be completed to a commuting square.
}
\end{defi}
Both \cite{K} and \cite{K1} deal with independence relations in AECats, where the former is only about simple independence and the latter introduces NSOP$_1$-like independence and deals with both. We will base ourselves on the terminology in \cite{K1}. The approach in both papers is to define an independence relation as a ternary relation on subobjects. Another approach is taken in \cite{LRV, LRV1} where an independence relation is seen as a class of commuting squares. The two approaches can be translated into one another. The former has the advantage of an easier to use independence calculus, whereas the latter has the advantage of admitting a more categorical treatment. We will therefore take the commuting squares perspective, and so we translate the relevant definitions into that perspective. We first define and discuss the various properties for an independence relation that is based on commuting squares, and then we briefly discuss how to compare this to the subobjects version (Remark \ref{subobj-vs-commuting-squares}).
\begin{defi}
\label{independence-relation}
{
\em
An independence relation $\nf$ on a category $\ck$ (usually an AECat with AP) is a class of commuting squares. If a commuting square like below is in the relation, we call it \emph{independent}. Often, we leave the morphisms involved unlabelled and we write $A \nf_C^M B$ if the square below is independent. Furthermore, we require the relation to be such that in the commuting diagram below we have $A \nf_C^M B$ iff $A \nf_C^N B$.
\[
\xymatrix@=1pc{
  A \ar[r] & M \ar[r] & N \\
  C \ar[r] \ar[u] & B \ar[u] &
}
\]
We call an independence relation \emph{basic} if it satisfies the following properties.
\begin{description}
\item[Invariance] Given two isomorphic commuting squares like below, the one square is independent iff the other square is independent.
\[
\xymatrix@=1pc{
  A' \ar[rrr] \ar@{=}[dr]|\cong & & & M' \\
  & A \ar[r] & M \ar@{=}[ur]|\cong & \\
  & C \ar[u] \ar[r] & B\ar@{=}[dr]|\cong \ar[u] & \\
  C' \ar[uuu]\ar@{=}[ur]|\cong \ar[rrr] & & & B' \ar[uuu]
}
\]
\item[Monotonicity] In the commuting diagram below $A \nf_C^M B$ implies $A \nf_C^M B'$.
\[
\xymatrix@=1pc{
  A \ar[rr] & & M \\
  C \ar[r] \ar[u] & B' \ar[r] & B \ar[u]
}
\]
\item[Transitivity] Independent squares can be composed, so in the commuting diagram below we have that if the two squares are independent then the outer rectangle is independent.
\[
\xymatrix@=1pc{
  A \ar[r] & M \ar[r] & N \\
  C \ar[r] \ar[u] & B \ar[r] \ar[u] & D \ar[u]
}
\]
\item[Symmetry] We have $A \nf_C^M B$ iff $B \nf_C^M A$.
\item[Existence] Any span can be completed to an independent square.
\end{description}
We also define the following further properties.
\begin{description}
\item[Base Monotonicity] Given a commuting diagram consisting of the solid arrows below such that $A \nf_C^M D$, we can find the dashed arrows such that everything commutes and $A' \nf_B^N D$.
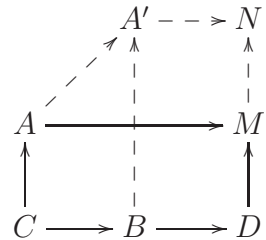
\[
\xymatrix@=2pc{
  & A' \ar@{-->}[r] & N \\
  A \ar@{-->}[ur] \ar[rr] & & M \ar@{-->}[u] \\
  C \ar[r] \ar[u] & B \ar[r] \ar@{-->}[uu] & D \ar[u]
}
\]
\item[Uniqueness] If in a commuting diagram consisting of the solid arrows below we have that if both $A \nf_C^M B$ and $A \nf_C^{M'} B$ then we can find the dashed arrows making everything commute.
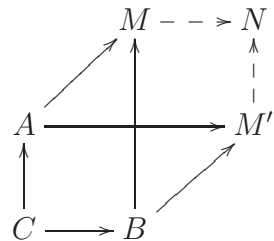
\[
\xymatrix@=2pc{
  & M \ar@{-->}[r] & N \\
  A \ar[rr] \ar[ur] & & M' \ar@{-->}[u] \\
  C \ar[u] \ar[r] & B \ar[uu] \ar[ur] &
}
\]
\item[$3$-amalgamation] Given a commuting diagram consisting of the solid arrows below (we call this a \emph{horn}) with every square independent, we can find the dashed arrows such that everything commutes and $A \nf_C^N N_3$.
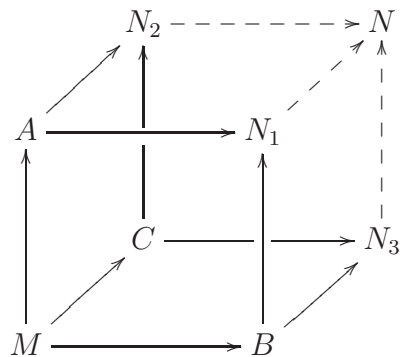
\[
\xymatrix@=2pc{
        & N_2 \ar @{-->}[rr] & & N\\
        A \ar[rr] \ar[ur] &
        & N_1 \ar @{-->}[ur] & \\
        & C \ar'[u][uu] \ar'[r][rr] & & N_3 \ar@{-->} [uu] & \\
        M \ar[ur]\ar[rr]\ar[uu] & & B \ar[ur]\ar[uu] &  
      }
\]
\end{description}
}
\end{defi}
Observe that an independence relation with the existence is closed
under $\sim$ (see \cite[Definition 3.4]{LRV}).

Given an independence relation $\nf$ with existence we have by \cite[Lemma 3.12]{LRV} that any commuting square
\[
\xymatrix@=1pc{
  B \ar[r] & D \\
  C \ar [u]^g \ar[r]_f & B \ar[u]
}
\]
with one of $f$ or $g$ an isomorphism is independent. In particular, if $f$ is an identity morphism the square is independent. Following \cite[Definition 3.16]{LRV} it thus makes sense to make the following definition.
\begin{defi}
{
\em
Given a category $\ck$ we write $\ck^2$ for the category that has as objects morphisms from $\ck$ and as morphisms the commuting squares in $\ck$. For an independence relation $\nf$ on $\ck$ that satisfies existence and transitivity we then define $\ck_{\nf}$ for the subcategory of $\ck^2$ where the morphisms are independent squares.
}
\end{defi}
We define two more properties for an independence relation $\nf$, which are properties of $\ck_{\nf}$ so we need to assume existence and transitivity.
\begin{defi}
\label{independence-relation-accessible-union}
{
\em
Let $\nf$ be an independence relation on $\ck$, satisfying existence and transitivity. We define the following properties.
\begin{description}
\item[Accessible] The category $\ck_{\nf}$ is accessible.
\item[Union] The category $\ck_{\nf}$ has directed colimits and these are preserved by the inclusion functor $\ck_{\nf} \hookrightarrow \ck^2$.
\end{description}
}
\end{defi}
We are now ready to translate the definition of the hierarchy of independence relations from \cite{K, K1} into our terminology. Once again, we refer to Remark \ref{subobj-vs-commuting-squares} for a comparison with the approach in \cite{K, K1}.
\begin{defi}
\label{independence-hierarchy}
{
\em
Let $\nf$ be a basic independence relation on $\ck$. We call $\nf$
\begin{itemize}
\item  a \emph{stable independence relation} if it is accessible and also satisfies union, $3$-amalgamation, base monotonicity and uniqueness.
\item  a \emph{simple independence relation} if it is accessible and also satisfies union, $3$-amalgamation and base monotonicity.
\item  an \emph{NSOP$_1$-like independence relation} if it is accessible and also satisfies union and $3$-amalgamation.
\end{itemize}
}
\end{defi}
We recall the canonicity theorems for the various notions of independence. For canonicity of NSOP$_1$-like independence we need the existence axiom for isi-forking, which we recall and deal with in Section \ref{sec:existence-axiom}.
\begin{theo}[{\cite[Theorems 1.1, 1.2, Corollary 1.3]{K1}}]
\label{canonicity}
Let $\ck$ be an AECat with AP.
\begin{enumerate}
\item There can be at most one simple independence relation on $\ck$.
\item If $\ck$ furthermore satisfies the existence axiom for isi-forking then there can be at most one NSOP$_1$-like independence relation on $\ck$.
\end{enumerate}
In particular, if there is a simple independence relation $\nf$ on $\ck$ then every NSOP$_1$-like independence relation on $\ck$ must coincide with $\nf$.
\end{theo}
Following \cite[Remark 2.6(2)]{LRV1} we also make the following definition.
\begin{defi}
{
\em
An independence relation $\nf$ that satisfies all the properties of a stable independence relation, except for possibly accessibility and union, is called \emph{weakly stable}.
}
\end{defi}
As usual, $3$-amalgamation follows from uniqueness. This was already shown in \cite[Proposition 6.16]{K}, and using Remark \ref{subobj-vs-commuting-squares} that would yield a proof of Proposition \ref{3-amalg-from-uniqueness} below. However, we still include a (cleaner) proof in our setting, because it is instructive. In particular this means that the $3$-amalgamation requirement in (weakly) stable independence relations is superfluous and so our notion of stable independence relation coincides with that of \cite{LRV, LRV1}, except that we also require the union property (in \cite{LRV1}, it is called $\aleph_0$-continuity).
\begin{propo}\label{3-amalg-from-uniqueness}
A basic independence relation with uniqueness also satisfies $3$-amalgamation.
\end{propo}
\begin{proof}
Consider a horn from \ref{independence-relation} and, using existence, find independent squares
\[
\vcenter{\vbox{\xymatrix@=1pc{
  N_1 \ar[r] & N' \\
  B \ar [u] \ar[r] & N_3 \ar[u]
}}}
\quad \text{and} \quad
\vcenter{\vbox{\xymatrix@=1pc{
  N_2 \ar[r] & N'' \\
  C \ar [u] \ar[r] & N_3 \ar[u]
}}}
\]
By transitivity, the squares 
\[
\vcenter{\vbox{\xymatrix@=1pc{
  A \ar[r] & N' \\
  M \ar [u] \ar[r] & N_3 \ar[u]
}}}
\quad \text{and} \quad
\vcenter{\vbox{\xymatrix@=1pc{
  A \ar[r] & N'' \\
  M \ar [u] \ar[r] & N_3 \ar[u]
}}}
\]
are independent. By uniqueness, they can be completed as follows
\[
\xymatrix@=2pc{
  & N'' \ar@{-->}[r] & N \\
  A \ar[ru] \ar[rr] & & N' \ar@{-->}[u] \\
  M \ar[u] \ar[r] & N_3 \ar[uu] \ar[ur] & \\
}
\]
This completes the starting horn to a cube whose diagonal square is independent.        
\end{proof}
\begin{rem}
\label{subobj-vs-commuting-squares}
{
\em
We compare the two different approaches to independence. Throughout this remark we fix an AECat $\ck$ with AP. In \cite{K, K1} an independence relation is a relation on triples of subobjects, whereas in \cite{LRV, LRV1} and in this paper we view an independence relation as a relation on commuting squares. We sketch how one can be recovered from the other, based on ideas from \cite[Definition 8.2]{LRV}.

First, suppose that $\nf$ is a basic independence relation in the subobject sense (basic means it satisfies \cite[Definition 4.5]{K1} without the union property, that is: invariance, monotonicity, transitivity, symmetry, existence and extension). We can define $\nf^{\square}$ by declaring a commuting square
\[
\xymatrix@=1pc{
  A \ar[r] & M \\
  C \ar[r] \ar[u] & B \ar[u]
}
\]
independent iff $A \nf_C^M B$, where $A, B, C \leq M$ are the subobjects represented by the morphisms in the square. It is a straightforward (but lengthy) check that $\nf^{\square}$ is a basic independence relation on commuting square in the sense of Definition \ref{independence-relation}.

Conversely, suppose that $\nf$ is a basic independence relation on commuting squares in the sense of Definition \ref{independence-relation}. Then we define an independence relation $\nf^\leq$ on subobjects as follows. For $A, B, C \leq M$ we set $A \nf_C^{\leq, M} B$ iff there are $M \to N$ and $A', B' \leq N$ with $A, C \leq A'$ and $B, C \leq B'$ such that
\[
\xymatrix@=1pc{
  A' \ar[r] & N \\
  C \ar[r] \ar[u] & B' \ar[u]
}
\]
is an $\nf$-independent square. Here the morphisms in the above square can be chosen to be any (compatible) representatives of the subobjects involved. Again, a straightforward (but lengthy) check shows that $\nf^\leq$ is a basic independence relation on subobjects in the sense of \cite[Definition 4.5]{K1} (without the union property).

These operations are inverse to one another, so for a basic independence relation we may take either perspective without losing anything. We can then further compare the other properties. Let $\nf$ be a basic independence relation, then
\begin{itemize}
\item $\nf$ satisfies base monotonicity iff it satisfies base monotonicity in the sense of \cite[Definition 4.9]{K1},
\item $\nf$ satisfies uniqueness iff it satisfies stationarity in the sense of \cite[Definition 4.9]{K1},
\item $\nf$ satisfies $3$-amalgamation iff it satisfies $3$-amalgamation in the sense of \cite[Definition 6.7]{K} iff it satisfies the independence theorem in the sense of \cite[Definition 4.9]{K1}.
\end{itemize}
Things become a little more subtle when we consider the union property and accessibility versus union and club local character in the sense of \cite[Definitions 4.5 and 4.9]{K1}. Let us weaken the definition of union in \cite[Definitions 4.5]{K1} to the following.
\begin{description}
\item[Weak union] Let $(B_i)_{i \in I}$ be a directed system with cocone $M$ and write $B = \colim_{i \in I} B_i$. Suppose furthermore that we have $C \leq A \leq M$ with $C \leq B_i$ for all $i \in I$. Then if $A \nf_C^M B_i$ for all $i \in I$, we have $A \nf_C^M B$.
\end{description}
The weakening here is the requirement that $C \leq A$ and $C \leq B_i$. However, this weaker property is still enough for the canonicity results in Theorem \ref{canonicity}, because the proofs can be adjusted so that at the start we can assume the base ($C$ in the above) to be included in all the other relevant subobjects (by the extension property). It is now straightforward (but lengthy) to check that if $\nf$ is accessible and satisfies union in the sense of Definition \ref{independence-relation-accessible-union} then it satisfies club local character (in the sense of \cite[Definition 4.9]{K1}) and weak union. The converse is not clear to us, but the above is enough for canonicity. This all becomes a lot easier when $\ck$ has joins of subobjects, which is almost always the case in this paper, see Remark \ref{joins-of-subobjects-make-things-easier}.

In summary, canonicity as in Theorem \ref{canonicity} does indeed go through with the definitions in as in this paper. Furthermore, given a stable/simple/NSOP$_1$-like independence relation in the sense of Definition \ref{independence-hierarchy} we get the corresponding independence relation on subobjects in the sense of \cite[Definition 4.10]{K1}, except that we possibly have to weaken it to have weak union.
}
\end{rem}
\begin{rem}
\label{joins-of-subobjects-make-things-easier}
{
\em
If our AECat $\ck$ has joins of subobjects, the entire comparison in Remark \ref{subobj-vs-commuting-squares} becomes a lot easier. This is because then for subobjects $A, B, C \leq M$ we have $A \nf_C^M B$ in the subobject sense iff
\[
\xymatrix@=1pc{
  A \vee C \ar[r] & M \\
  C \ar[r] \ar[u] & B \vee C \ar[u]
}
\]
is independent in the commuting square sense. One then easily sees that the property ``weak union'' from Remark \ref{subobj-vs-commuting-squares} is equivalent to the property ``union'' from \cite[Definitions 4.5]{K1} (modulo the properties of a basic independence relation).

In almost all our cases the AECat of interest has joins of subobjects. This is because as soon as the AECat is locally multipresentable we have joins of subobjects: namely for subobjects $A, B \leq M$ take the wide pullback of all subobjects of $M$ that $A$ and $B$ factor through. Due to Lemma \ref{lmp1} working with a factorization system, like in Theorem \ref{simple}, will yield a locally multipresentable AECat. In Theorems \ref{lmp-nsop} and \ref{lmp-simple} being locally multipresentable is simply an assumption. Theorem \ref{nsop} is our only main result where the AECat may not be locally multipresentable. 
}
\end{rem}

\section{$\NSOP_1$-like independence}
\begin{defi}\label{coherently-amalgamable-squared}
{
\em
Let $\cm$ be a class of monomorphisms in a locally presentable category $\ck$ that is closed under composition and contains all isomorphisms. We call $\cm$ \emph{coherently amalgamable} if it satisfies the following axioms:
\begin{itemize}
\item[(A1)] if $fg \in \cm$ and $f \in \cm$ then $g \in \cm$ (this is sometimes called \emph{coherence}),
\item[(A2)] the pushout of two morphisms in $\cm$ is in $\cm$.
\end{itemize}
We call $\cm$ \emph{squared} if it is coherently amalgamable and additionally satisfies:
\begin{itemize}
\item[(A3)] the pullback of two morphisms in $\cm$ is in $\cm$,
\item[(A4)] the pushout of two morphisms in $\cm$ is a pullback.
\end{itemize}
}
\end{defi}

\begin{rem}
{
\em
We note two things in the context of Definition \ref{coherently-amalgamable-squared}.
\begin{enumerate}
\item Condition (A2) means that, given a pushout square
  $$  
    \xymatrix@=1pc{
      C \ar[r]^{\bar{f}} & P \\
      A \ar [u]^g \ar[r]_f & B \ar[u]_{\bar{g}}
    }
    $$
in $\ck$ with $f,g\in\cm$, then $\bar{f},\bar{g}\in\cm$. This is weaker than $\cm$ being closed under pushouts (see Example \ref{binary-functions}), which means that if $f\in\cm$ then $\bar{f}\in\cm$.

\item By \cite{Ri} (A4) is equivalent, modulo (A1)--(A3), to the condition that every epimorphism in $\cm$ is an isomorphism.
\end{enumerate}
}
\end{rem}
\begin{defi}
{
\em
Given a class $\cm$ in a category $\ck$ that is closed under composition and contains all isomorphisms, we write $\ck_\cm$ for the subcategory of $\ck$ whose objects are those in $\ck$ and whose morphisms are precisely those from $\cm$.
}
\end{defi}
Note that if $\cm$ is coherently amalgamable then $\ck_\cm$ has AP.
\begin{defi}
\label{m-effective-square}
{
\em
Similar to the cellular squares from \cite{LRV1} we define the commuting square of morphisms in $\cm$ below on the left to be \emph{$\cm$-effective} if the induced map $P \to D$ from the pushout below on the right is in $\cm$
$$  
    \xymatrix@=1pc{
      C \ar[r] & D \\
      A \ar [u] \ar[r] & B \ar[u]
    }
    \quad
    \xymatrix@=1pc{
      C \ar[r] & P \\
      A \ar [u] \ar[r] & B \ar[u]
    }
$$
}
\end{defi}
Note that if $\cm$ is squared then $\cm$-effective squares are pullback squares.
\begin{propo}
\label{weakly}
Suppose that $\cm$ is a coherently amalgamable class of monomorphisms in a locally presentable category $\ck$. Then the $\cm$-effective squares form a weakly stable independence relation on $\ck_\cm$.
\end{propo}
\begin{proof}
Exactly as in \cite[Theorem 2.7, Remark 2.8]{LRV1} where only (A2) is needed (and not
the closedness of $\cm$ under pushouts).
\end{proof}
\begin{propo}\label{very}
Suppose that $\cm$ is a squared class of monomorphisms in a locally presentable category $\ck$. Then the pullback squares form a basic independence relation on $\ck_\cm$.
\end{propo}
\begin{proof}
Straightforward and exactly as in \cite[Theorem 5.1]{LRV}.
\end{proof}
\begin{rem}
{
\em
Note that in both Propositions \ref{weakly} and \ref{very} we do not need the full assumption of $\ck$ being locally presentable, having pushouts and pullbacks of (co)spans of morphisms in $\cm$ is enough.
}
\end{rem}

\begin{defi}\label{cubic}
{
\em
A squared class $\cm$ of monomorphisms will be called \emph{cubic} if,
given a horn in $\cm$ whose bottom, front and left-side squares are pullbacks
$$ 
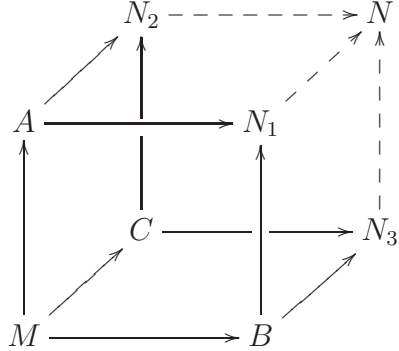
\xymatrix@=2pc{
	& N_2 \ar @{-->}[rr]^{} & & N\\
	A \ar[rr]^>>>>>>>>>>{}\ar [ur]^{} &
	& N_1 \ar @{-->}[ur]_{} & \\
	& C\ar'[u][uu]^{}\ar'[r][rr] & & N_3\ar @{-->} [uu]_{} & \\
	M\ar[ur]^{}\ar[rr]\ar[uu]^{} & & B\ar[ur]_{}\ar[uu]_>>>>>>>>>{} &  
}
$$
its colimit has all edges in $\cm$ and the diagonal square 
$$  
\xymatrix@=1pc{
  A \ar[r] & N \\
  M \ar [u] \ar[r] & N_3 \ar[u]
}
$$
is a pullback.
}
\end{defi}

\begin{rem}
{
\em
Definition \ref{cubic} implies (A4). It suffices to take a horn with $M = C = B$ and $A = N_1 = N_2$, and identity morphisms between those objects. Then the colimit is a pushout of $A \leftarrow M \to N_3$, which is also the diagonal square. This shows that the condition about the diagonal square is needed.
}
\end{rem}
\begin{question}
{
\em
Does the condition on the diagonal square in Definition \ref{cubic} follow from being squared and the cocone in Definition \ref{cubic} being in $\cm$?
}
\end{question}
The following is immediate from the definitions.
\begin{propo}\label{cubic1}
If $\cm$ is a cubic class in a locally presentable category $\ck$ then the independence relation on $\ck_\cm$ given by pullback squares satisfies $3$-amalgamation.
\end{propo}

\begin{rem}
{
\em
A horn in Definition \ref{cubic} is $2$-dimensional. A $1$-dimensional horn is a span $A \leftarrow M \to B$. Hence a $2$-dimensional analogue of Definition \ref{coherently-amalgamable-squared}(A4) is that a colimit of a horn in Definition \ref{cubic} has all edges in $\cm$ the resulting cube is a limit. This means that $M$ is a limit of the right upper corner horn   
$$ 
\xymatrix@=2pc{
        & N_2 \ar[rr]^{h_0} & & N\\
        A \ar[rr]^>>>>>>>>>>>>>>>>{e_1}\ar[ur]^{e_0} &
        & N_1 \ar[ur]_{h_1} & \\
        & C\ar'[u][uu]_{f_0}\ar'[r][rr]^{f_1} & & N_3\ar[uu]_{h_2} & \\
        M\ar@{-->}[ur]^{c}\ar@{-->}[rr]_b\ar@{-->}[uu]^{a} & & B\ar[ur]_{g_1}\ar[uu]_>>>>>>>>>>>>>>{g_0} &  
      }
$$
This follows from $\cm$ being cubic. Indeed, let $u:X\to A$, $v:X\to C$ and $w:X\to B$ be a cone to the right upper corner horn. Then
$$
h_0e_0u=h_1e_1u=h_1g_0w=h_2g_1w.
$$
Since the diagonal square is a pullback, there is a unique $t:X\to M$ such that $at=u$ and $g_1bt=g_1w$. Hence $bt=w$ and 
$$
f_1ct=g_1bt=g_1w=f_1v.
$$
Thus $ct=v$, which implies that $M$ is a limit.

On the other hand, it is easy to see that if $M$ is a limit of $h_i:N_i\to N$, $i=1,2,3$ then $\cm$ is cubic. This is different from the situation before though, where $M$ is the limit of a diagram including $A, B, C$.
}
\end{rem}
\begin{defi}
\label{continuous-accessible}
{
\em
A class $\cm$ of morphisms that is closed under composition and contains all isomorphisms in a category $\ck$ will be called
\begin{enumerate}
\item[(A5)] \emph{continuous} if $\ck_\cm$ is closed under directed colimits in $\ck$ and
\item[(A6)] \emph{accessible} if $\ck_\cm$ is an accessible category.
\end{enumerate}
}
\end{defi}
 
\begin{rem}\label{aec}
{
\em
We note two things in the context of Definition \ref{continuous-accessible}.
\begin{enumerate}
\item Condition (A5) means that, given a directed diagram 
$(k_{ij}:K_i\to K_j)_{i\leq j\in I}$ in $\ck_\cm$, its colimit 
$k_i:K_i\to K$ in $\ck$ exists, all $k_i$ belong to $\cm$ and, for a cocone $l_i:K_i\to L$ with $l_i$ in $\cm$, the induced morphism $K\to L$ belongs to $\cm$.

\item If $\cm$ is a coherently amalgamable, continuous and accessible class of monomorphisms in a locally presentable category $\ck$ then $\ck_\cm$ is an AECat with AP. If $\ck$ is locally finitely presentable then $\ck_\cm$ is in fact (equivalent to) an AEC (abstract elementary class) with AP \cite[Corollary 5.7]{BR}.
\item Conditions (A1), coherence, and (A5) together imply that if an object $A$ is $\lambda$-presentable in $\ck$ (where $\ck$ is any category, not necessarily locally presentable) then it also is $\lambda$-presentable in $\ck_\cm$. This follows immediately from writing out definitions.
\end{enumerate}
}
\end{rem}

\begin{repeated-theorem}[Theorem \ref{nsop}]
Suppose that $\ck$ is a locally finitely presentable category equipped with a cubic, continuous and accessible class $\cm$ of monomorphisms. Then $\ck_\cm$ has an $\NSOP_1$-like independence relation given by pullback squares.
\end{repeated-theorem}
\begin{proof}
Following Propositions \ref{very} and \ref{cubic1}, pullback squares give a basic independence relation that also satisfies $3$-amalgamation. Since $\cm$ is continuous $\ck_\cm$ has directed colimits and so $(\ck_\cm)^2$ has directed colimits, so to prove the union property it is enough to show that $\ck_{\nf}$ is closed in $(\ck_\cm)^2$ under directed colimits. Let $(f_i: K_i \to L_i)_{i \in I}$ be a directed diagram in $\ck_{\nf}$ with morphisms $(k_{ij},l_{ij}): f_i \to f_j$ for $i \leq j$. So for all $i \leq j$ the square
\begin{align}
\label{colimiting-cocone}
\begin{split}
\xymatrix@=1pc{
    L_i \ar[r]^{l_{ij}} & L_j \\
    K_i \ar[u]^{f_i} \ar[r]_{k_{ij}} &
    K_j \ar[u]_{f_j}
}
\end{split}
\end{align}
is a pullback. Let $((k_i,l_i): f_i \to f)_{i \in I}$ be the colimiting cocone in $(\ck_\cm)^2$. Fix $i \in I$, then the square
$$
\xymatrix@=1pc{
    L_i \ar[r]^{l_i} & L \\
    K_i \ar[u]^{f_i} \ar[r]_{k_i} &
    K \ar[u]_{f}
}
$$
is a pullback, because it is a directed colimit of pullbacks as in (\ref{colimiting-cocone}) (for $i \leq j$), and pullbacks commute with directed colimits in $\ck$ (see \cite[Proposition 1.59]{AR}). We thus see that the cocone $((k_i,l_i): f_i \to f)_{i \in I}$ is in $\ck_{\nf}$. Now let $((p_i, q_i): f_i \to g)_{i \in I}$ be another cocone such that for all $i \in I$ the square
\begin{align}
\label{other-cocone}
\begin{split}
\xymatrix@=1pc{
    L_i \ar[r]^{q_i} & Q \\
    K_i \ar[u]^{f_i} \ar[r]_{p_i} &
    P \ar[u]_{g}
}
\end{split}
\end{align}
is a pullback, and let $(p,q): f \to g$ be the corresponding universal morphism in $(\ck_\cm)^2$. Then
$$
\xymatrix@=1pc{
    L \ar[r]^{q} & Q \\
    K \ar[u]^{f} \ar[r]_{p} &
    P \ar[u]_{g}
}
$$
is a pullback because it is a directed colimit of pullbacks as in (\ref{other-cocone}). We conclude that $f$ is indeed the colimit of $(f_i)_{i \in I}$ in $\ck_{\nf}$.

It remains to prove that $\ck_{\nf}$ is accessible. First, we will show that there is a regular cardinal $\mu$ such that, given $K\to L$ in $\cm$ with $L$ $\mu$-presentable, then $K$ is $\mu$-presentable. Let $\lambda$ be such that $\ck_\cm$ is $\lambda$-accessible. There is a regular cardinal $\kappa\geq\lambda$ such that, given $K\to L$ in $\cm$ with $L$ $\lambda$-presentable, then $K$ is $\kappa$-presentable. It suffices to take $\mu \geq \kappa$ and $\mu \unrhd \lambda$. Then $\mu$-presentable objects are $\mu$-small $\lambda$-directed colimits in $\ck_\cm$ of $\lambda$-presentable objects (see \cite[Proposition 2.3.11]{MP}). Indeed, let $K \to L$ be in $\cm$ and $L$ be $\mu$-presentable. We express $L$ as a $\mu$-small directed colimit $(l_i: L_i\to L)_{i \in I}$ in $\ck_\cm$ of $\lambda$-presentable objects. Form pullbacks
$$
 \xymatrix@=1pc{
        K \ar@{}\ar[r]^{} & L \\
        K_i \ar [u]^{k_i} \ar [r]_{} &
        L_i \ar[u]_{l_i}
      }
      $$
Then the $K_i$ are $\mu$-presentable and $(k_i: K_i \to K)_{i \in I}$ is a $\mu$-small $\lambda$-directed colimit (because pullbacks commute with directed colimits in $\ck$). Thus $K$ is $\mu$-presentable.
 
Now, consider a morphism  $f: K \to L$ in $\cm$. Express $L$ as a $\mu$-directed colimit of $\mu$-presentable objects ($\ck_\cm$ is $\mu$-accessible by our choice of $\mu$, see \cite[Theorem 2.3.10]{MP}) $L_i$ in $\ck_\cm$. First form the pullbacks on the left below. The universal property then yields morphisms $k_{ij}: K_i \to K_j$ for all $i \leq j$ such that the square on the right below commutes and is a pullback.
$$
\xymatrix@=1pc{
    L_i \ar[r]^{l_i} & L \\
    K_i \ar[u]^{f_i} \ar[r]_{k_i} &
    K \ar[u]_{f}
}
\quad
\xymatrix@=1pc{
    L_i \ar[r]^{l_{ij}} & L_j \\
    K_i \ar[u]^{f_i} \ar[r]_{k_{ij}} &
    K_j \ar[u]_{f_j}
}
$$  
This makes $f$ a $\mu$-directed colimit of $f_i$ in $\ck_{\nf}$. Since $K_i$ are $\mu$-presentable the $f_i$ are $\mu$-presentable in $(\ck_\cm)^2$, and so they are in $\ck_{\nf}$ too because the inclusion functor preserves directed colimits and $\nf$ seen as a class of morphisms in $(\ck_\cm)^2$ is left-cancellable. We conclude that $\ck_{\nf}$ is $\mu$-accessible.  
\end{proof} 

\begin{rem}\label{nsop1}
{
\em
In Theorem \ref{nsop}, we could only assume that $\ck$ is locally presentable and directed colimits of morphisms from $\cm$ commute with pullbacks. 

This generalization applies to Grothendieck toposes or Grothendieck abelian categories and regular monomorphisms.
}
\end{rem}
\begin{rem}\label{regular-monos-almost-squared}
{
\em
If the class $\cm$ of regular monomorphisms in a locally finitely presentable category $\ck$ is closed under composition then it is continuous and accessible (see the proof of \cite[Lemma 4.3]{LRV}). If furthermore condition (A2) in Definition \ref{coherently-amalgamable-squared} holds, so the pushout of two regular monomorphisms consists of regular monomorphisms, then $\cm$ is also squared. So in particular, when $\ck$ is a coregular category then the class of regular monomorphisms is squared, continuous and accessible. The converse is not true, see Example \ref{binary-functions}.
}
\end{rem}
\begin{exam}\label{graphs}
{
\em
The category $\Gra$ of graphs (= sets equipped with a symmetric binary relation) is locally finitely presentable and the class $\cm$ of embeddings (= regular monomorphisms) is cubic, continuous and accessible. Thus pullback squares form an $\NSOP_1$-like independence relation on $\Gra_\cm$. Since $\Gra$ does not have effective unions, this independence is not stable.
}
\end{exam}
\begin{exam}\label{binary-functions}
{
\em
Let $\BinFunc$ be the category of binary functions $f: X \times X \to Y$. Morphisms $f_1 \to f_2$ are pairs $(u,v)$ such  that the square
$$
\xymatrix@=2pc{
X_1\times X_1\ar[r]^{u\times u} \ar[d]_{f_1} & X_2\times X_2\ar[d]^{f_2} \\
Y_1 \ar[r]_{v} & Y_2
}
$$
commutes. As $\BinFunc$ is a category of algebras in a fixed finitary (multi-sorted) signature, it is locally finitely presentable (see \cite[Remark 3.4(6)]{AR}). Every monomorphism is regular, and these are exactly the pairs $(u,v)$ such that $u$ and $v$ are injective.

Pushouts of monomorphisms are described as follows. Let $f_i: X_i \times X_i \to Y_i$ be binary functions for $i \in \{0, 1, 2\}$, where $f_0$ is a subfunction of $f_1$ and $f_2$. We will construct the pushout of the inclusions $f_1 \leftarrow f_0 \to f_2$. We may assume $X_1 \setminus X_0$ and $X_2 \setminus X_0$ to be disjoint, and the same for $Y_1 \setminus Y_0$ and $Y_2 \setminus Y_0$. Set $X = X_1 \cup X_2$ and $Y = Y_1 \cup Y_2 \cup [(X_1 \setminus X_0) \times (X_2 \setminus X_0)] \cup [(X_2 \setminus X_0) \times (X_1 \setminus X_0)]$, and define $f: X \times X \to Y$ as
\[
f(x, x') = \begin{cases}
f_1(x, x') & \text{if } x, x' \in X_1 \\
f_2(x, x') & \text{if } x, x' \in X_2 \\
(x, x') & \text{else}
\end{cases}
\]
It is then straightforward to check that the inclusions $f_1 \to f \leftarrow f_2$ form the pushout.

So the class $\cm$ of regular monomorphisms is closed under composition and pushouts of morphisms in $\cm$ are again in $\cm$. So by Remark \ref{regular-monos-almost-squared} we have that $\cm$ is squared, continuous and accessible. Using a construction similar to the pushout above one indeed verifies that $\cm$ is also cubic.

We also show that $\BinFunc$ is not coregular (or cellular), establishing that the regular monomorphisms forming a squared class is indeed weaker than being coregular. This weakening of assumptions compared to \cite{LRV1} is relevant, because our results apply to $\BinFunc_\cm$.

Let $X_0 = \{a, b \}$, $Y_0 = (X_0 \times X_0) \cup \{\heartsuit, \clubsuit\}$ and let $f_0: X_0 \times X_0 \to Y_0$ be the inclusion. Take $X_1 = \{ a, b, c \}$ and $Y_1 = Y_0$, let $f_1: X_1 \times X_1 \to Y_1$ extend $f_0$ by setting $f_1(c, a) = f_1(a, c) = f(c, c) = \heartsuit$ and $f_1(c, b) = f(b, c) = \clubsuit$. Finally, let $X_2 = \{d\}$ and $Y_2 = \{ *, \heartsuit, \clubsuit \}$ and let $f_2: X_2 \times X_2$ be the constant function with value $*$.

We let $(u_1, v_1): f_0 \to f_1$ be the inclusion and $(u_2, v_2): f_0 \to f_2$ is given by $u_2(a) = u_2(b) = d$ and $v_2$ maps all of $X_0 \times X_0$ to $*$ and is the identity on $\{\heartsuit, \clubsuit\}$. Now suppose that we have the following commuting diagram:
$$  
  \xymatrix@=2pc{
    f_1 \ar[r]^{(u_2', v_2')} & f \\
    f_0 \ar [u]^{(u_1, v_1)} \ar [r]_{(u_2, v_2)} & f_2 \ar[u]_{(u_1', v_1')}
  }
$$
Then we calculate
\begin{align*}
v'_1(\heartsuit) =
v'_2(f_1(a, c)) =
f(u'_1 u_2(a), u'_2(c)) =
f(u'_1 u_2(b), u'_2(c)) =
v'_2(f_1(b, c)) =
v'_1(\clubsuit),
\end{align*}
and we see that $v_1'$ is not injective and hence not a regular monomorphism.
}
\end{exam}
Those familiar with model theory may recognize Examples \ref{graphs} and \ref{binary-functions} as the theories of the random graph \cite[Exercise 3.3.1]{TZ} and the generic binary function \cite[Section 3]{KR} respectively. Indeed, these are the theories of the existentially closed models of $\Gra_\cm$ and $\BinFunc_\cm$ respectively.
\begin{exams}\label{not-cubic}
{
\em
We consider some examples of locally finitely presentable categories $\ck$ with a squared, continuous and accessible class of monomorphisms $\cm$ that is not cubic (in fact: $3$-amalgamation fails). Hence, in each case, pullback squares in $\cm$ just form a basic independence relation, which is continuous and accessible. In each case $\cm$ will be the class of regular monomorphisms, making many of these properties automatic, see Remark \ref{regular-monos-almost-squared}. The remaining two properties---that $\cm$ is closed under composition and pushouts of morphisms in $\cm$ are again in $\cm$---will each time be straightforward to verify.
\begin{enumerate}
\item Take $\ck = \Pos$, the category of posets, and let $\cm$ be the class of embeddings (= regular monomorphisms). We show that $\cm$ is not cubic. Consider the diagram with $M = \emptyset, A = \{a\}, B = \{b\}, C = \{c\}, N_1 = \{a < b\}, N_2 = \{c < a\}, N_3 = \{b < c\}$. In any cocone $N$ we get $a \leq b \leq c \leq a$, so they all need to be identified, hence none of the morphisms $N_i \to N$ (for $i \in \{1,2,3\}$) can be a monomorphism. The same happens in the category of Boolean algebras.

\item In the category $\Grp$ of groups, regular monomorphisms coincide with monomorphisms and hence with injective group homomorphisms (see e.g.\ \cite[Exercise 7H]{AHH}). So we take $\cm$ to be the class of injective group homomorphisms.  The fact that the pushout of two monomorphisms consists again of monomorphisms is well known and is often called the ``amalgamated free product'' (see e.g.\ \cite[Theorem 2]{S}), so $\cm$ is indeed squared, continuous and accessible. However, $\cm$ is not cubic. To see this we can use the construction in \cite[Proposition 4.1]{SU} or the following similar construction. Consider the diagram where $M$ is the trivial group, $A = \langle a \rangle$, $B = \langle b, b' \rangle$, $C = \langle c, c' \rangle$, $N_1 = \langle a, b, b' : a b a^{-1} = b' \rangle$, $N_2 = \langle a, c, c' : a c a^{-1} = c' \rangle$ and $N_3 = \langle b, b', c, c' : b'c' = c'b' \rangle$ with the obvious embeddings between these groups. Here $\langle \bar{x} : \bar{R} \rangle$ denotes the group presented with generators $\bar{x}$ and relators $\bar{R}$. Any cocone $N$ for this diagram will have that $bc = cb$, however $bc \neq cb$ in $N_3$, so the morphism $N_3 \to N$ cannot be a monomorphism.

We note that this example heavily relies on non-commuting elements. In fact, in the category of abelian groups the pullback squares of regular monomorphisms form a stable independence relation \cite[Example 4.8(1) and Theorem 5.1]{LRV}.
\end{enumerate}
}
\end{exams}

\section{Simple independence}
Weakly stable independence is base monotonic because base monotonicity follows from existence and uniqueness. By essentially replacing the uniqueness property by 3-amalgamation we thus also lose the base monotonicity property. In this section we give a characterization in terms of factorization systems of when base monotonicity can be recovered.

\begin{defi}
{
\em
A coherently amalgamable class $\cm$ of monomorphisms in a locally presentable category $\ck$ will be called \emph{strongly coherently amalgamable} if there is a class $\ce$ of epimorphisms such that $(\ce,\cm)$ is a factorization system.
}
\end{defi}
\begin{fact}[{\cite{KELLY}}]
\label{regular-factorization-system}
In a category with cokernel pairs and equalizers of cokernel pairs (e.g.\ any locally presentable category), the following are equivalent:
\begin{enumerate}
\item the composition of two regular monomorphisms is a regular monomorphism,
\item regular monomorphisms coincide with strong monomorphisms,
\item for any morphism $f$, if $f = m e$ is its factorization through the equalizer $m$ of its kernel pair then $e$ is an epimorphism,
\item taking $\ce$ and $\cm$ to be the classes of epimorphisms and regular monomorphisms respectively, we have that $(\ce, \cm)$ is a factorization system.
\end{enumerate}
\end{fact}
\begin{proof}
The equivalence of (1), (2) and (3) is Propositions 2.7 and 3.8 in \cite{KELLY}. Then clearly (4) implies (2), while (1) and (3) together imply (4).
\end{proof}
\begin{fact}[{\cite[Proposition 2.1.4]{FK}}]
{
\em
If $(\ce, \cm)$ is a factorization system in any category where $\ce$ is a class of epimorphisms then $\cm$ is left-cancellable. In particular, $\cm$ is coherent.
}
\end{fact}
The following shows that $\cm$ being part of a factorization system in an accessible category brings it very close to being accessible itself (in the sense of Definition \ref{continuous-accessible}).
\begin{propo}
\label{accessibility-of-continuous-factorization}
{
\em
Suppose that $(\ce, \cm)$ is a factorization system in a $\lambda$-accessible category $\ck$ with $\cm$ a continuous class. Then the following are equivalent:
\begin{enumerate}
\item $\ck_\cm$ is $\lambda$-accessible,
\item $\cm$ is accessible (i.e.\ $\ck_\cm$ is accessible),
\item for any object $A$ in $\ck$ there is a set of $\ce$-images (up to isomorphism).
\end{enumerate}
In particular, the above conditions are fulfilled if $\ce$ is a class of epimorphisms and $\ck$ is co-wellpowered, e.g.\ when $\ck$ has pushouts or if there is a proper class of strongly compact cardinals.
}
\end{propo}
\begin{proof}
Without assuming (1), (2) or (3) we can follow the proof of \cite[Proposition 1.69(i)]{AR}, to see that $\ce$-images of $\kappa$-presentable objects in $\ck$ are $\kappa$-presentable in $\ck_\cm$, for any $\kappa$.

We now use this to prove the equivalence of (1), (2) and (3). The implication (1) $\Rightarrow$ (2) is immediate. For (2) $\Rightarrow$ (3) we pick $\kappa$ such that $A$ is $\kappa$-presentable in $\ck$, so by the above any $\ce$-image of $A$ is $\kappa$-presentable in $\ck_\cm$. The result then follows because in an accessible category there is only a set of $\kappa$-presentable objects and the category is locally small. Finally, for (3) $\Rightarrow$ (1) we follow the proof strategy from \cite[Lemma 4.3]{LRV}. By continuity we have that $\ck_\cm$ has directed colimits. For an object $K$ from $\ck_\cm$ we write it as a $\lambda$-directed colimit of $\lambda$-presentable objects $(A_i)_{i \in I}$ in $\ck$. For each coprojection $a_i: A_i \to K$ we get an $(\ce, \cm)$-factorisation $a_i = a_i'' a_i'$, making $(a_i'': B_i \to K)_{i \in I}$ into a $\lambda$-directed colimit in $\ck_\cm$. By the above the $B_i$ are $\lambda$-presentable in $\ck_\cm$ and by assumption there is only a set of such objects. We thus conclude that $\ck_\cm$ is indeed $\lambda$-accessible.

The final sentence follows because accessible categories with pushouts are co-wellpowered \cite[Theorem 2.49]{AR}, and if there is a proper class of strongly compact cardinals any accessible category is co-wellpowered \cite[Proposition 6.2.2 and Theorem 6.3.8]{MP}.
\end{proof}
\begin{rem}
\label{regular-strongly-coherently-amalgamable}
{
\em
Let $\cm$ be the class of regular monomorphisms in a locally finitely presentable category. If $\cm$ is closed under composition then by Fact \ref{regular-factorization-system} we have that $(\ce, \cm)$ is a factorization system, where $\ce$ is the class of epimorphisms. In particular, if the pushout of two regular monomorphisms consists of regular monomorphisms then following Remark \ref{regular-monos-almost-squared} we have that $\cm$ is strongly coherently amalgamable, squared, continuous and accessible.
}
\end{rem}

\begin{defi}
{
\em
Given a factorization system $(\ce,\cm)$ in a category with pushouts, a commutative square below on the left will be called an $\ce$-\emph{pushout} if the induced morphism $P \to D$ from the pushout below on the right is in $\ce$.
$$
\xymatrix@=1pc{
        B \ar[r] & D \\
        A \ar[u] \ar[r] &
        C \ar[u]
      }
\quad
\xymatrix@=1pc{
        B \ar[r] & P \\
        A \ar[u] \ar[r] &
        C \ar[u]
      }
$$ 
}
\end{defi}

\begin{defi}\label{strongly-def}
{
\em
A strongly coherently amalgamable class $\cm$ of monomorphisms will be called \emph{strongly squared} if it is squared and, given a commutative diagram  
$$  
  \xymatrix@=1pc{
    M_1 \ar[r]^{} & M_3 \ar[r]^{}  & M_5 \\
    M_0 \ar [u]^{} \ar [r]_{} & M_2 \ar[u]_{} \ar[r]_{} & M_4 \ar[u]_{}
  }
  $$
of monomorphisms from $\cm$ where the left square is an $\ce$-pushout and outer rectangle a pullback, then the right square is a pullback. 

A \emph{strongly cubic} class of monomorphisms is a cubic class which is strongly squared.
}
\end{defi}

\begin{propo}\label{strongly}
A strongly coherently amalgamable and squared class $\cm$ of monomorphisms is strongly squared if and only if the independence relation given by pullback squares is base monotonic.
\end{propo}
\begin{proof}
($\Rightarrow$) Let $\cm$ be a strongly squared class of monomorphisms. 
Consider the commutative diagram  
 $$  
  \xymatrix@=1pc{
    M_1 \ar[rr]^{} &  & M_4 \\
    M_0 \ar [u]^{} \ar [r]_{} & M_2 \ar[r]_{} & M_3 \ar[u]_{}
  }
  $$
of morphisms from $\cm$ such that the outer rectangle is a pullback. Form a pushout
$$
 \xymatrix@=1pc{
        M_1 \ar@{}\ar[r]^{} & M_1' \\
        M_0 \ar [u]^{} \ar [r]_{} &
        M_2 \ar[u]_{}
      }
      $$
We get a commutative diagram 
$$  
  \xymatrix@=1pc{
    M_1 \ar[r]^{} & M_1' \ar[r]^{t}  & M_4 \\
    M_0 \ar [u]^{} \ar [r]_{} & M_2 \ar[u]_{} \ar[r]_{} & M_3 \ar[u]_{}
  }
  $$
where $t$ is the induced morphism. Let 
$$
M_1' \xrightarrow{\ t_1\ } M_1'' \xrightarrow{\ t_2\ } M_4  
$$ 
be the $(\ce,\cm)$-factorization of $t$. In a commutative diagram
$$  
  \xymatrix@=1pc{
    M_1 \ar[r]^{} & M_1'' \ar[r]^{t_2}  & M_4 \\
    M_0 \ar [u]^{} \ar [r]_{} & M_2 \ar[u]_{} \ar[r]_{} & M_3 \ar[u]_{}
  }
  $$   
the left square is an $\ce$-pushout and the outer rectangle is a pullback.   
Since $\cm$ is left-cancellable, the diagram consists of monomorphisms
from $\cm$. Since $\cm$ is strongly squared, the right square is a pullback. Hence the independence given by pullbacks is base monotonic.

($\Leftarrow$) Let the independence given by pullbacks is base monotonic. Consider
a commutative diagram  
$$  
  \xymatrix@=1pc{
    M_1 \ar[r]^{} & M_3 \ar[r]^{}  & M_5 \\
    M_0 \ar [u]^{} \ar [r]_{} & M_2 \ar[u]_{} \ar[r]_{} & M_4 \ar[u]_{}
  }
  $$
of morphisms from $\cm$ where the left square is an $\ce$-pushout and outer rectangle a pullback. Applying base monotonicity we get a commutative diagram
  \[
  \xymatrix@=2pc{
    & M_1' \ar@{-->}[r]_{} & M_5' \\
    M_1 \ar@{-->}[ur]^{} \ar[rr]|>>>>>>>>>{} &  & M_5 \ar@{-->}[u]_{} \\
    M_0 \ar [u]^{} \ar [r]_{} & M_2 \ar[r]_{} \ar@{-->}[uu]|>>>>>>{} & M_4 \ar[u]_{}
  }
  \]
of morphisms in $\cm$ where the right rectangle $(M_2,M_1',M_4,M_5')$ is a pullback. Let
$$
\xymatrix@=1pc{
        M_1 \ar@{}\ar[r]^{} & P \\
        M_0 \ar [u]^{} \ar [r]_{} &
        M_2 \ar[u]_{}
      }
      $$
be a pushout and $P\to M_3$ and $P\to M_1'$ be the induced morphisms. By the uniqueness of induced morphisms we have that
$$
\xymatrix@=1pc{
        M_1' \ar@{}\ar[r]^{} & M_5' \\
        P \ar [u]^{} \ar [r]_{} &
        M_3 \ar[u]_{}
      }
      $$
commutes. Since $P \to M_3$ is in $\ce$ and $M'_1 \to M'_5$ is in $\cm$, there is the diagonal $M_3\to M_1'$ such that both triangles commute. By left-cancellability we have that the diagonal $M_3 \to M_1'$ is in $\cm$. Chasing diagrams and using the fact that $M_1' \to M_5'$ is a monomorphism we see that $M_2 \to M_3 \to M_1'$ is the same as our original morphism $M_2 \to M_1'$. Since
$$
\xymatrix@=1pc{
M_1' \ar [r]^{}  & M_5' \\
M_3 \ar[r]^{} \ar [u]^{}  & M_5 \ar [u]_{}\\
M_2 \ar [r]{} \ar [u]^{} & M_4 \ar [u]_{}
}
$$ 
commutes, $M_1' \to M_3$ is a monomorphism and the outer rectangle is a pullback we have that the lower square is a pullback. Hence $\cm$ is strongly squared.
\end{proof}

\begin{repeated-theorem}[Theorem \ref{simple}]
Suppose that $\ck$ is a locally finitely presentable category equipped with a strongly cubic and continuous class $\cm$ of monomorphisms. Then $\ck_\cm$ has a simple independence relation given by pullback squares.
\end{repeated-theorem}
\begin{proof}
It follows from Theorem \ref{nsop} and Propositions \ref{strongly} and \ref{accessibility-of-continuous-factorization}.
\end{proof} 

\begin{coro}\label{effective}
Let $\ck$ be a locally finitely presentable coregular category. Then the following are equivalent:
\begin{enumerate}
\item $\ck$ has effective unions (i.e.\ pullback squares are effective),
\item regular monomorphisms are strongly cubic and cofibrantly generated,
\item regular monomorphisms are cubic and cofibrantly generated.
\end{enumerate}
\end{coro}
\begin{proof}
(1) $\implies$ (2). Let $\ck$ be a locally finitely presentable coregular category with effective unions. Following \cite[Proposition 1.12]{B}, regular monomorphisms are cofibrantly generated. Since pullback squares are effective, they form a weakly stable independence which satisfies $3$-amalgamation by Proposition \ref{3-amalg-from-uniqueness}. It follows regular monomorphisms are cubic: 3-amalgamation guarantees a suitable cocone for every horn and since $\cm$ is left-cancellable and by general properties of pullbacks, the colimiting cocone satisfies all required properties. Since weakly stable independence is base monotonic, regular monomorphisms are strongly cubic by Proposition \ref{strongly}.

(2) $\implies$ (3). Trivial.

(3) $\implies$ (1). Assume that regular monomorphisms are cubic and cofibrantly generated. Following Remark \ref{regular-monos-almost-squared} we can apply Theorem \ref{nsop} to see that pullback squares form an NSOP$_1$-like independence relation. Following \cite[Theorem 3.1]{LRV1}, effective pullback squares form a stable independence relation. By canonicity (Corollary \ref{canonicity-in-main-thms}) we have that the two coincide. Hence pullback squares are effective.
\end{proof}

\begin{rem}\label{nsop2}
{
\em
Like in Theorem \ref{nsop}, we could only assume that $\ck$ is locally presentable and directed colimits of morphisms from $\cm$ commute with pullbacks. 
}
\end{rem}

\begin{exam}
\label{graphs-strongly-cubic}
{
\em
The class $\cm$ of embeddings in $\Gra$ is strongly cubic (recall from Example \ref{graphs} that it is cubic), corresponding to the well-known fact that the theory of the random graph is simple unstable \cite[Remark 8.2.4]{TZ}. The factorization system $(\ce,\cm)$ has $\ce$ consisting of surjections (on vertices). Consider a commutative diagram of embeddings
$$  
  \xymatrix@=1pc{
    M_1 \ar[r]^{} & M_3 \ar[r]^{}  & M_5 \\
    M_0 \ar [u]^{} \ar [r]_{} & M_2 \ar[u]_{} \ar[r]_{} & M_4 \ar[u]_{}
  }
  $$
where the left square is an $\ce$-pushout and outer rectangle a pullback.
Take two pushouts
$$  
  \xymatrix@=1pc{
    M_1 \ar[r]^{} & P \ar[r]^{}  & Q \\
    M_0 \ar [u]^{} \ar [r]_{} & M_2 \ar[u]_{} \ar[r]_{} & M_4 \ar[u]_{}
  }
  $$
Since the forgetful functor $\Gra\to\Set$ preserves both pushouts and pullbacks and $\Set$ has effective unions, the induced morphism $Q\to M_5$
is injective. Thus the induced morphism $P\to M_3$ is injective on vertices, hence bijective on vertices (because it is in $\ce$).

Just considering sets of vertices, and assuming the arrows are genuine embeddings, we thus have $M_3 = M_1 \cup M_2$ (by the above) and so $M_3 \cap M_4 = (M_1 \cup M_2) \cap M_4 = (M_1 \cap M_4) \cup (M_2 \cap M_4) = M_0 \cup M_2 = M_2$. We conclude that the right square $(M_2,M_3,M_4,M_5)$ is a pullback.
}
\end{exam}

\begin{exam}
{
\em
We consider the category of binary functions $\BinFunc$ from Example \ref{binary-functions}, with $\cm$ the class of regular monomorphisms. There is a factorization system $(\ce, \cm)$ where $\ce$ consists of epimorphisms (i.e., both $u$ and $v$ are surjective). So $\cm$ is strongly coherently amalgamable and it is cubic. We will show that it is not strongly squared, corresponding to the fact that the theory of the generic binary function is NSOP$_1$ but not simple \cite[Proposition 3.14]{KR}. Let $X = \{ a, b \}$ and $Y = \{y, z\}$ and consider the function $f: X \times X \to Y$ given by $f(a, a) = y$ and $f(b, b) = f(a, b) = f(b, a) = z$. For $0 \leq i \leq 4$ consider the subfunctions $f_i: X_i \times X_i \to Y_i$ given by:
\begin{align*}
&X_0 = \emptyset, &Y_0 = \{y\}, \\
&X_1 = \{a\}, &Y_1 = \{y\}, \\
&X_2 = \{b\}, &Y_2 = \{z\}, \\
&X_3 = \{a, b\}, &Y_3 = \{y, z\}, \\
&X_4 = \{b\}, &Y_4 = \{y, z\}.
\end{align*}
The inclusions then fit in a commutative diagram as below:
$$  
  \xymatrix@=1pc{
    f_1 \ar[r]^{} & f_3 \ar[r]^{}  & f \\
    f_0 \ar [u]^{} \ar [r]_{} & f_2 \ar[u]_{} \ar[r]_{} & f_4 \ar[u]_{}
  }
$$
Then the outer rectangle is a pullback and in the left square $f_3$ is the epimorphic image of the pushout $f_1 \leftarrow f_0 \to f_2$. However, the right square is not a pullback.
}
\end{exam}

\begin{exams}\label{posets}
{
\em
We continue Example \ref{not-cubic} of locally finitely presentable categories with $\cm$ the class of regular monomorphisms, such that $\cm$ is continuous, accessible and squared but not cubic. Following Fact \ref{regular-factorization-system} $\cm$ is in fact part of a factorization system $(\ce, \cm)$ where $\ce$ is the class of epimorphisms. In this example we will show that being strongly squared is independent from these earlier properties.
\begin{enumerate}
\item The class $\cm$ of embeddings of $\Pos$ is strongly squared. Epimorphisms are precisely the surjections. Since a pushout with $f,g \in \cm$
$$
 \xymatrix@=1pc{
        B \ar@{}\ar[r]^{\bar{f}} & D \\
        A \ar [u]^{g} \ar [r]_{f} &
        C \ar[u]_{\bar{g}}
      }
      $$  
is a pullback and $\bar{f},\bar{g}$ are jointly surjective, the forgetful functor $\Pos\to\Set$ preserves pushouts of embeddings. Consider
$$  
  \xymatrix@=1pc{
    M_1 \ar[r]^{} & M_3 \ar[r]^{}  & M_5 \\
    M_0 \ar [u]^{} \ar [r]_{} & M_2 \ar[u]_{} \ar[r]_{} & M_4 \ar[u]_{}
  }
  $$
where the left square is an $\ce$-pushout and the outer rectangle is a pullback. We will show that the right square is a pullback.

Take $a\in M_3\cap M_4$. Let 
$$
 \xymatrix@=1pc{
        M_1 \ar@{}\ar[r]^{} & P \\
        M_0 \ar [u]^{} \ar [r]_{} &
        M_2 \ar[u]_{}
      }
      $$  
be a pushout and $p:P\to M_3$ the induced morphism. There is $a'\in P$ such that $p(a')=a$. If $a'\in M_2$ then $a\in M_2$. If $a'\in M_1$ then $a\in M_0$, hence $a\in M_2$.

\item Embeddings in $\Grp$ are not strongly squared. Consider
$$  
  \xymatrix@=1pc{
    \langle a\rangle \ar[r]^{} & \langle a,b\rangle \ar[r]^{}  & \langle a,b\rangle \\
    0 \ar [u]^{} \ar [r]_{} & \langle b\rangle \ar[u]_{} \ar[r]_{} & \langle b,aba\rangle \ar[u]_{}
  }
  $$
Then the left square is a pushout, the outer rectangle is a pullback but the right square is not a pullback.
\end{enumerate}
}
\end{exams}

\section{Independence in locally multipresentable categories}
We recall the definition and equivalent characterisations of locally multipresentable categories (see e.g.\ \cite[Definition 4.28 and Theorem 4.30]{AR}).
\begin{defi}
{
\em
A \emph{multicolimit} of a diagram $D$ is a (possibly empty) set of cocones, such that every cocone factors uniquely through a unique cocone in that set. We call a $\lambda$-accessible category $\ck$ \emph{locally $\lambda$-multipresentable} if the following equivalent conditions hold:
\begin{enumerate}
\item $\ck$ has all small multicolimits,
\item $\ck$ has all small connected limits.
\end{enumerate}
We call $\ck$ \emph{locally multipresentable} if it is locally $\lambda$-multipresentable for some $\lambda$.
}
\end{defi}
\begin{lemma}\label{lmp1}
Let $\cm$ be a strongly coherently amalgamable continuous class of morphisms in a locally $\lambda$-presentable category $\ck$. Then $\ck_{\cm}$ is a locally $\lambda$-multipresentable category with AP.
\end{lemma}
\begin{proof}
By Proposition \ref{accessibility-of-continuous-factorization} we have that $\ck_\cm$ is $\lambda$-accessible. To show that $\ck_\cm$ has connected limits it suffices to show it has equalizers and wide pullbacks, which easily follows from $(\ce,\cm)$ being a proper factorization system (meaning that $\ce$ consists of epimorphisms). 
\end{proof}
\begin{rem}\label{gen}
{
\em
In the context Lemma \ref{lmp1} we have that the multipushout of a span $B \leftarrow A \to C$ consists of all its $\ce$-pushouts.
}
\end{rem}
We adjust the definitions of coherently amalgamable, squared and cubic to work with multicolimits.
\begin{defi}\label{multi-coherently-amalgamable-squared}
{
\em
Let $\cm$ be a class of monomorphisms in a locally multipresentable category $\ck$ that is closed under composition and contains all isomorphisms. We call \emph{coherently multiamalgamable} if it satisfies the following axioms:
\begin{itemize}
\item[(A1)] if $fg \in \cm$ and $f \in \cm$ then $g \in \cm$ (this is sometimes called \emph{coherence}),
\item[(A2')] the multipushout of two morphisms in $\cm$ contains at least one square that is entirely in $\cm$.
\end{itemize}
We call $\cm$ \emph{multisquared} if it is coherently multiamalgamable and also satisfies:
\begin{itemize}
\item[(A3)] the pullback of two morphisms in $\cm$ is in $\cm$,
\item[(A4')] the multipushout of two morphisms in $\cm$ contains at least one square that is entirely in $\cm$ and which is also a pullback.
\end{itemize}
Finally, we call $\cm$ \emph{multicubic} if it is multisquared and the multicolimit of the horn in Definition \ref{cubic} contains a cocone that is entirely in $\cm$ and where the diagonal square is a pullback.
}
\end{defi}
\begin{rem}
{
\em
We can generalize Remark \ref{aec}(2) to the following: if $\cm$ is a coherently multiamalgamable, continuous and accessible class of monomorphisms then $\ck_\cm$ is an AECat with AP. The weakening of (A2) to (A2') still implies AP.
}
\end{rem}

The following theorem generalizes \ref{nsop}.
\begin{repeated-theorem}[Theorem \ref{lmp-nsop}]
Suppose that $\ck$ is a locally finitely multipresentable category equipped with a multicubic, continuous and accessible class $\cm$ of monomorphisms. Then $\ck_\cm$ has an $\NSOP_1$-like independence relation given by pullback squares.
\end{repeated-theorem}
\begin{proof}
We follow the proof of Theorem \ref{nsop}. In the proof of accessibility of $\ck_{\nf}$ we use the commutation of directed colimits with pullbacks (see \cite[Theorem 4.30]{AR}).
\end{proof}

\begin{exam}
\label{locally-multi-presentable-examples}
{
\em
Let $K$ be a set and $\ck$ the category of functions $d:X\times X\to K$. Hence $\ck$ is the category of sets equipped with a $K$-valued distance $d$. Morphisms $(X_1,d_1)\to (X_2,d_2)$ are distance preserving mappings, i.e., mappings $u:X_1\to X_2$ such that $d_2(u(x),u(y))=d_1(x,y)$ for every $x,y\in X$.

The category $\ck$ is accessible because it is the inserter category $\Ins(-^2,C_K)$ where $C_K$ is the constant functor with the value $K$ (see \cite[Theorem 2.72]{AR}). Since $C_K$ preserves all non-empty limits, it preserves connected limits and thus $\ck$ has connected limits (cf.\ \cite[Exercise 2j(1)]{AR}). Thus $\ck$ is locally multipresentable. In fact, it is easy to see that $\ck$ is locally finitely multipresentable where finitely presentable objects are finite sets with a $K$-valued distance. Let $\cm$ be the class of regular monomorphisms which coincide with injective morphisms.

Multipushouts are pushouts of underlying sets with all possible choices of cross-distances. Hence every instance of a multipushout of two regular monomorphisms is a pullback, so $\cm$ is multisquared. Similarly, we show that $\cm$ is multicubic. Following Theorem \ref{lmp-nsop}, $\ck_\cm$ has an $\NSOP_1$-like independence given by pullbacks.

We note that if the distance function is a metric then distance preserving maps no longer form multicubic class. For example, take $K = \mathbb{R}$ and consider the category of metric spaces with isometries (i.e., distance preserving maps). This is a locally finitely multipresentable category (see e.g.\ \cite[Theorem 5.9]{LRV-univ}). In the notation of Definition \ref{cubic}, consider the horn consisting of embeddings between $M = \emptyset$, $A = \{a\}$, $B = \{b\}$, $C = \{c\}$, $N_1 = \{a, b\}$, $N_2 = \{a, c\}$ and $N_3 = \{b, c\}$ with distance functions such that $d(a, b) = d(a, c) = 1$ and $d(b, c) = 3$. Due to the triangle inequality, there cannot be a cocone $N$ of this horn, as the map $N_3 \to N$ cannot preserve the distance.
}
\end{exam}

\begin{defi}
{
\em
Given a factorization system $(\ce, \cm)$ in a category with multipushouts, a commutative square below on the left will be called an \emph{$\ce$-multipushout} if the induced morphism $P \to D$ from an instance of the multipushout below on the right is in $\ce$.
$$
 \xymatrix@=1pc{
        B \ar[r] & D \\
        A \ar [u] \ar [r] &
        C \ar[u]
      }
\quad
 \xymatrix@=1pc{
        B \ar[r] & P \\
        A \ar[u] \ar[r] &
        C \ar[u]
      }
$$
}
\end{defi}
\begin{defi}
{
\em
A strongly coherently multiamalgamable class $\cm$ of monomorphisms will be called \emph{strongly multisquared} if it is multisquared and, given a commutative diagram  
$$  
  \xymatrix@=1pc{
    M_1 \ar[r]^{} & M_3 \ar[r]^{}  & M_5 \\
    M_0 \ar [u]^{} \ar [r]_{} & M_2 \ar[u]_{} \ar[r]_{} & M_4 \ar[u]_{}
  }
  $$
of monomorphisms from $\cm$ where the left square is an $\ce$-multipushout and outer rectangle a pullback, then the right square is a pullback. 

A \emph{strongly multicubic} class of monomorphisms is a multicubic class which is strongly multisquared.
}
\end{defi}
\begin{propo}
\label{lmp-strongly}
A strongly coherently multiamalgamable and multisquared class $\cm$ of monomorphisms is strongly multisquared if and only if the independence relation given by pullback squares is base monotonic.
\end{propo}
\begin{proof}
We largely follow the proof of Proposition \ref{strongly} and sketch where some minor modifications need to be made.

($\Rightarrow$) Instead of taking a pushout, we let $M_1'$ be the instance of the multipushout that admits a morphism $t: M_1' \to M_4$. The remainder of the proof is then as written.

($\Leftarrow$)
Instead of taking the pushout
$$
\xymatrix@=1pc{
        M_1 \ar[r] & P \\
        M_0 \ar[u] \ar[r] &
        M_2 \ar[u]
      }
      $$
we take the instance of the multipushout that admits a morphism $P \to M_3$. Since the squares     
\[
\vcenter{\vbox{
\xymatrix@=1pc{
    M_1 \ar[r] & M_3 \\
    M_0 \ar[u] \ar[r] &
    M_2 \ar[u]
}
}}
\quad
\text{and}
\quad
\vcenter{\vbox{
\xymatrix@=1pc{
    M_1 \ar[r] & M_1' \\
    M_0 \ar[u] \ar[r] &
    M_2 \ar[u]
}
}}
\]
are amalgated by $M_5'$, there is also $P \to M_1'$. The remainder of the proof is then as written.
\end{proof}

The following theorem generalizes Theorem \ref{simple}.

\begin{repeated-theorem}[Theorem \ref{lmp-simple}]
Suppose that $\ck$ is a locally finitely multipresentable category equipped with a strongly multicubic, continuous and accessible class $\cm$ of monomorphisms. Then $\ck_\cm$ has an simple independence relation given by pullback squares.
\end{repeated-theorem}
\begin{proof}
It follows from Theorem \ref{lmp-nsop} and Proposition \ref{lmp-strongly}.
\end{proof}

\begin{exams}
{
\em
We give two examples where Theorem \ref{lmp-simple} applies.
\begin{enumerate}
\item The class $\cm$ of regular monomorphisms in the category of sets equipped with $K$-valued distances that we considered in Example \ref{locally-multi-presentable-examples} is in fact strongly cubic. So by Theorem \ref{lmp-simple} the pullback squares form a simple independence relation.

\item Fix some field $K$ and let $\Bil_K$ be the category of bilinear spaces over $K$ with injective linear maps that respect the bilinear form (if the reader wishes they can further restrict to symmetric or alternating bilinear spaces). In \cite[Theorem 1.1]{K2} it is established that linear independence yields a simple independence relation on $\Bil_K$. We can reproduce this result by applying Theorem \ref{lmp-simple} to the class of all morphisms in $\Bil_K$, which is straightforwardly verified to be strongly multicubic.
\end{enumerate}
}
\end{exams}

\section{The existence axiom}
\label{sec:existence-axiom}
In \cite[Section 6]{K1} various notions of independence are defined. Most important to us are the notions of long dividing, isi-dividing and isi-forking. The negations of these notions then yield independence relations $\nf^\ld$, $\nf^\isid$ and $\nf^\isif$ respectively. The goal of this section is to prove the following.
\begin{theo}\label{existence-axiom}
Let $\ck$ be an AECat with AP. Suppose that $\nf$ is a weakly stable independence relation on $\ck$. Then $\ck$ satisfies the existence axiom for $\nf^\isif$. That is, any commuting square
\[
\xymatrix@=1pc{
    A \ar[r] & M \\
    C \ar[u] \ar[r] & B \ar[u]
}
\]
of morphisms in $\cm$ where $C \to B$ is an isomorphism is $\nf^\isif$-independent.
\end{theo}

\begin{rem}\label{existence-terminology-clash}
{\em
There is a slight clash in terminology. In \cite{K1} the existence property for an independence relation $\nf$ means that any square where the bottom morphism is an isomorphism, is $\nf$-independent. So the existence axiom for $\nf^\isif$ is then the assumption that $\nf^\isif$ has the existence property in this sense, see Theorem \ref{existence-axiom}.

In this paper we use the terminology from \cite{LRV}, where the existence property refers to being able to complete any span to an independent square. This is sometimes called ``full existence'' and is of course stronger, see \cite[Lemma 3.12]{LRV}.
}
\end{rem}

\begin{coro}\label{existence-axiom-presentable}
Let $\ck$ be a locally finitely presentable category and suppose that $\cm$ is a coherently amalgamable, continuous and accessible class of monomorphisms. Then $\ck_\cm$ satisfies the existence axiom for isi-forking.
\end{coro}
\begin{proof}
By Remark \ref{aec} $\ck_\cm$ is an AECat (even an AEC) with AP. By Proposition \ref{weakly} $\ck_\cm$ has a weakly stable independence relation. So we can apply Theorem \ref{existence-axiom}.
\end{proof}
\begin{repeated-theorem}[Corollary \ref{canonicity-in-main-thms}]
The independence relations in Theorems \ref{nsop}, \ref{simple} and \ref{lmp-simple} are canonical in the following sense. Given a category $\ck$ and class of monomorphisms $\cm$ satisfying the conditions in one of these theorems, then any NSOP$_1$-like independence relation on $\ck_\cm$ is given by pullback squares.
\end{repeated-theorem}
\begin{proof}
We apply Theorem \ref{canonicity} in each case. For Theorems \ref{simple} and \ref{lmp-simple} the application is immediate. For Theorem \ref{nsop} we use Corollary \ref{existence-axiom-presentable} to see that we have the existence axiom for isi-forking.
\end{proof}
The only main result where we have not established canonicity for the independence relation is Theorem \ref{lmp-nsop}. This is because we need the existence axiom for isi-forking there, which cannot be established in the same way as we did for the locally finitely presentable case (i.e.\ by constructing a weakly stable independence relation and applying Theorem \ref{existence-axiom}), as the following example shows.
\begin{exam}
{
\em
We give an example of a locally finitely multipresentable category that cannot have a weakly stable independence relation and whose morphisms are all regular monomorphisms, which in turn form a strongly multicubic, continuous and accessible class of monomorphisms (so in particular: coherently multiamalgamable). We note that this also means that Theorem \ref{lmp-simple} applies, so we have a simple independence relation and hence the existence axiom for isi-forking, and so this is also an example of the independence between these properties and having a weakly stable independence relation.

Let $\TwoGraph$ be the category of two-graphs with graph embeddings. We recall that a \emph{two-graph} consists of a set of vertices $X$ together with a 3-ary edge relation $E \subseteq [X]^3$ (so edges are undirected and do not contain duplicate vertices), such that between any four vertices there is an even number of edges. Then $\TwoGraph$ is a locally finitely multipresentable category because it can be axiomatized by a finitary disjunctive theory \cite[Exercise 5f]{AR} (add a relation symbol for the negation of the edge relation and enumerate all two-graphs on four elements). Furthermore, it has AP by \cite[Propositions 3.2 and 3.3]{MeirPapa}.

We claim that $\TwoGraph$ cannot have a weakly stable independence relation. Suppose that $\nf$ is weakly stable. By \cite[Fact 2.8 and Propositions 3.2 and 3.3]{MeirPapa} there is, up to isomorphism, a unique countable two-graph $G$. We apply existence for $\nf$ to the span $\{d\} \supseteq \emptyset \subseteq G$, where $\{d\}$ is two-graph with one vertex, to find a two-graph $H$ containing $d$ and $G$ and such that $d \nf^H_\emptyset G$. Pick some distinct $a$ and $b$ in $G$. We consider two cases. In the first case we assume that there is no edge between $a,b,d$. By uniqueness of $G$ and amalgamation we can find $c \in G$ distinct from $a$ and $b$ such that there is an edge between $a,b,c$. Let $f: \{a, b\} \to G$ be the embedding $f(a) = a$ and $f(b) = c$ and consider the following diagram of solid arrows (unnamed arrows are inclusions)
\[
\xymatrix@=2pc{
  & H \ar@{-->}[r]^g & H' \\
  \{d\} \ar[rr] \ar[ur] & & H \ar@{-->}[u]_h \\
  \emptyset \ar[u] \ar[r] & \{a,b\} \ar[uu] \ar[ur]_f &
}
\]
Both squares consisting of solid arrows are $\nf$-independent by monotonicity. So by uniqueness for $\nf$ we find the dashed arrows making everything commute. Since there is no edge between $a, b, d$ there is no edge between $g(a), g(b), g(d)$. By commutativity of the diagram we have $(g(a), g(b), g(d)) = (hf(a), hf(b), h(d)) = (h(a), h(c), h(d))$. We thus see that there is no edge between $a, c, d$. Using an analogous argument we also see that there is no edge between $b, c, d$. We thus conclude that in the four vertices $a, b, c, d$ the only edge is between $a, b, c$, which is not allowed in a two-graph. The second case where there is an edge between $a, b, d$ is similar, we now just pick $c$ such that there is no edge between $a, b, c$. Both cases lead to a contradiction, so $\nf$ cannot exist.

The above example is essentially a version of the model-theoretic example that types over the empty set in theory of two-graphs do not extend to global invariant types \cite{C}. It is also well-known that the first-order theory of the two-graph is simple. In fact, taking $\cm$ to be the class of all morphisms in $\TwoGraph$ (which are all regular monomorphisms) it is easy to check that $\cm$ is strongly multisquared and by \cite[Remark 7.4.4]{Marimon} we also have that it is multicubic. Therefore, pullback squares form a simple independence relation by Theorem \ref{lmp-simple}.
}
\end{exam}

The remainder of this section is dedicated to the proof of Theorem \ref{existence-axiom}, so in the remainder of this section we work in an AECat $\ck$ with AP.

The following is simplified definition of Galois types in AECats \cite[Definition 3.3]{K}.
\begin{defi}\label{galois-type}
{\em
Let $A \leftarrow C \to B$ be a span of morphisms. Suppose that we have two commuting squares
\[
\vcenter{\vbox{
\xymatrix@=1pc{
    A \ar[r]^a & M \\
    C \ar[u] \ar[r] & B \ar[u]_b
}
}}
\quad
\text{and}
\quad
\vcenter{\vbox{
\xymatrix@=1pc{
    A \ar[r]^{a'} & M' \\
    C \ar[u] \ar[r] & B \ar[u]_{b'}
}
}}
\]
where we label the morphisms $C \to M$ and $C \to M'$ with $c$ and $c'$ respectively. Then we say that these squares have the same \emph{Galois type}, and write $\gtp(a,b,c; M) = \gtp(a', b', c'; M')$, if there is a cospan $M \to N \leftarrow M'$ such that the resulting diagram commutes
\[
\xymatrix@=1pc{
    & M \ar[r] & N \\
    A \ar[ru]^(.3){a} \ar[rr]_(.3){a'} & & M' \ar[u] \\
    C \ar[u] \ar[r] & B \ar[uu]_(.3){b}  \ar[ur]_(.3){b'} & \\
}
\]
}
\end{defi}
Note that having the same Galois type is exactly the relation $\sim$ from \cite[Definition 3.2]{LRV}. This is an equivalence relation, where transitivity follows from the assumption that $\ck$ has AP.

The definitions of $\nf^\ld$, $\nf^\isid$ and $\nf^\isif$ are quite technical. For our purposes we can work with a simplified definition.
\begin{defi}\label{dividing-notions}
{\em
Let
\[
\xymatrix@=1pc{
    A \ar[r]^a & M \\
    C \ar[u]^{c_a} \ar[r]_{c_b} & B \ar[u]_b
}
\]
be a commuting square, where the morphism $C \to M$ is labelled $c$.
\begin{enumerate}
\item We say that the square \emph{does not simplified long divide} and call it $\nf^\sld$-independent if the following holds. There is a proper class of regular cardinals $\lambda$ such that given a morphism $f: M \to N$ and $\lambda$ many morphisms $(b_i: B \to N)_{i < \lambda}$, such that $fc = b_i c_b$ for each $i < \lambda$, we can find a cospan $N \xrightarrow{g} N' \xleftarrow{a'} A$ with $\gtp(a', g b_i, gfc; N') = \gtp(a, b, c; M)$ for all $i < \lambda$.
\item We say that the square \emph{does not simplified long fork} and call it $\nf^\slf$-independent if the following holds. For every $f: M \to N$ and any two sets of morphisms $\{a_j: A \to N\}_{j \in J}$ and $\{d_j: D_j \to N\}_{j \in J}$ such that $fc$ factors through each $d_j$, we have that (i) implies (ii) below.
\begin{itemize}
\item[(i)] For any cospan $N \xrightarrow{g} N' \xleftarrow{a'} A$ with $\gtp(a', g f b, gfc; N') = \gtp(a, b, c; M)$, there is $j \in J$ such that $\gtp(a', g d_j, gfc; N') = \gtp(a_j, d_j, f c; N)$.
\item[(ii)] There is some $j \in J$ such that
\[
\xymatrix@=1pc{
    A \ar[r]^{a_j} & N \\
    C \ar[u]^{c_a} \ar[r] & D_j \ar[u]_{d_j}
}
\]
is $\nf^\sld$-independent.
\end{itemize}
\end{enumerate}
}
\end{defi}
\begin{lemma}\label{sld-slf-facts}
We have that $\nf^\sld$-independence implies $\nf^\ld$-independence, which in turn implies $\nf^\isid$-independence. It follows that $\nf^\slf$-independence implies $\nf^\isif$-independence.
\end{lemma}
\begin{proof}
The entire point of introducing $\nf^\sld$ and $\nf^\slf$ is that we do not have to introduce the more involved definitions of the other independence notions. However, the claims in this lemma are just writing out definitions, where the claims in the first sentence are used in the second.
\end{proof}
\begin{propo}\label{weakly-stable-implies-sld}
Suppose we have a weakly stable independence relation $\nf$ and a square as in Definition \ref{dividing-notions} that is $\nf$-independent. Then that square is $\nf^\sld$-independent.
\end{propo}
\begin{proof}
Let $f: M \to N$ and $(b_i: B \to N)_{i < \lambda}$ be as in Definition \ref{dividing-notions}(1). Apply existence for $\nf$ to $M \xleftarrow{b} B \xrightarrow{fb} N$ to find $M \xrightarrow{h} N' \xleftarrow{g} N$, such that the resulting square in $\nf$-independent. By transitivity for $\nf$ the outer rectangle below is $\nf$-independent:
\[
\xymatrix@=1pc{
    A \ar[r]^{a} & M \ar[r]^{h} & N' \\
    C \ar[u]^{c_a} \ar[r]_{c_b} & B \ar[u]_{b} \ar[r]_{fb} & N \ar[u]_{g}
}
\]
Set $a' = ha$, we verify that we have constructed the required cospan $N \xrightarrow{g} N' \xleftarrow{a'} A$ from Definition \ref{dividing-notions}(1). So let $i < \lambda$. The bottom arrow in the above diagram is $f b c_b = fc = b_i c_b$. So by monotonicity for $\nf$ the following square is $\nf$-independent:
\[
\xymatrix@=1pc{
    A \ar[r]^{a'} & N' \\
    C \ar[u]^{c_a} \ar[r]_{c_b} & B \ar[u]_{g b_i}
}
\]
Uniqueness for $\nf$ is then precisely saying that $\gtp(a', gb_i, gfc; N') = \gtp(a, b, c; M)$, which is what we needed to show.
\end{proof}
\begin{propo}\label{weakly-stable-implies-slf}
Suppose we have a weakly stable independence relation $\nf$ and a square as in Definition \ref{dividing-notions} that is $\nf$-independent. Then that square is $\nf^\slf$-independent.
\end{propo}
\begin{proof}
Let $f: M \to N$ together with $\{a_j: A \to N\}_{j \in J}$ and $\{d_j: D_j \to N\}_{j \in J}$ be as in Definition \ref{dividing-notions}(2) and assume that (i) there holds for these sets. Exactly like in the proof of Proposition \ref{weakly-stable-implies-sld} we obtain the rectangle diagram where the outer rectangle is $\nf$-independent, and like there we set $a' = ha$. Then by construction $\gtp(a', gfb, gfc; N') = \gtp(a', hb, hc; N') = \gtp(a, b, c; M)$, so by our assumption (i) there is $j \in J$ such that $\gtp(a', gd_j, gfc; N') = \gtp(a_j, d_j, fc; N)$. As the bottom arrow in rectangle diagram is $fc$, which factors through $d_j$, we get by monotonicity of $\nf$ that
\[
\xymatrix@=1pc{
    A \ar[r]^{a'} & N' \\
    C \ar[u]^{c_a} \ar[r] & D_j \ar[u]_{g d_j}
}
\]
is $\nf$-independent. Hence by the equality of Galois types that followed from (i) we get that
\[
\xymatrix@=1pc{
    A \ar[r]^{a} & N \\
    C \ar[u]^{c_a} \ar[r] & D_j \ar[u]_{d_j}
}
\]
is $\nf$-independent, and thus $\nf^\sld$-independent by Proposition \ref{weakly-stable-implies-sld}.
\end{proof}
\begin{proof}[Proof of Theorem \ref{existence-axiom}]
Any commuting square where the bottom is an isomorphism is $\nf$-independent. So by Proposition \ref{weakly-stable-implies-slf} that square is $\nf^\slf$-independent, and hence it is $\nf^\isif$-independent by Lemma \ref{sld-slf-facts}.
\end{proof}

\section*{Acknowledgements}
We thank the anonymous referee for their feedback, which helped improve the presentation of this paper.

\end{document}